\documentclass[12pt,letterpaper]{article}

\usepackage{amsmath,amssymb,amsfonts,amsthm}

\usepackage[colorlinks=true,linkcolor=black,citecolor=black,urlcolor=blue,pdfborder={0 0 0}]{hyperref}
\usepackage{graphicx}
\usepackage{color}
\usepackage[font=sf, labelfont={sf,bf}, margin=1cm]{caption}
\usepackage[margin=39mm]{geometry}
\usepackage{bbm}
\usepackage{cleveref}

\crefname{theorem}{Theorem}{Theorems}
\crefname{thm}{Theorem}{Theorems}
\crefname{mainthm}{Theorem}{Theorems}
\crefname{lemma}{Lemma}{Lemmas}
\crefname{lem}{Lemma}{Lemmas}
\crefname{remark}{Remark}{Remarks}
\crefname{prop}{Proposition}{Propositions}
\crefname{defn}{Definition}{Definitions}
\crefname{corollary}{Corollary}{Corollaries}
\crefname{section}{Section}{Sections}
\crefname{figure}{Figure}{Figures}

\newtheorem{thm}{Theorem}[section]
\newtheorem{theorem}[thm]{Theorem}
\newtheorem{mainthm}{Theorem}
\newtheorem{lemma}[thm]{Lemma}
\newtheorem{corollary}[thm]{Corollary}
\newtheorem{prop}[thm]{Proposition}

\newtheorem{question}[thm]{Question}

\theoremstyle{definition}
\newtheorem{defn}[thm]{Definition}

\renewcommand{\P}{\mathbb P}
\newcommand{\Z}{\mathbb Z}
\newcommand{\E}{\mathbb E}
\newcommand{\R}{\mathbb R}
\newcommand{\N}{\mathbb N}
\newcommand{\eps}{\varepsilon}
\newcommand{\Var}{\operatorname{Var}}
\newcommand{\LIS}{\operatorname{LIS}}

\newcommand{\iI}{\mathcal{I}}
\DeclareMathOperator{\ord}{ord}

\numberwithin{equation}{section}
\title{Increasing subsequences of random walks}

\author{%
	Omer Angel \thanks{University of British Columbia, Vancouver BC, V6T 1Z2, Canada. E-mail: {\tt
			angel@math.ubc.ca}. Supported in part by NSERC.}
	\and Rich\'ard Balka \thanks{University of British Columbia, and Pacific Institute for the Mathematical Sciences, Vancouver BC, V6T 1Z2, Canada. E-mail: {\tt balka@math.ubc.ca}. 
		Former affiliations: Department of Mathematics, University of Washington,
		and Alfr\'ed R\'enyi Institute of Mathematics, Hungarian Academy of Sciences. Supported by the National Research, Development and Innovation Office-NKFIH, 104178.}
	\and Yuval Peres\thanks{Microsoft Research, 1 Microsoft Way, Redmond, WA 98052, USA. E-mail: {\tt peres@microsoft.com}.}}

\begin{document}
\maketitle

\renewcommand\thefootnote{}
\footnote{{\em 2010 Mathematics Subject Classification}: 60G17, 60G50.}
\footnote{{\em Keywords}: random walk, restriction, monotone, increasing
  subsequence.}

\begin{abstract}
  Given a sequence of $n$ real numbers $\{S_i\}_{i\leq n}$, we consider the
  longest weakly increasing subsequence, namely $i_1<i_2<\dots <i_L$ with
  $S_{i_k} \leq S_{i_{k+1}}$ and $L$ maximal. When the elements $S_i$ are
  i.i.d.\ uniform random variables, Vershik and Kerov, and Logan and Shepp
  proved that $\E L=(2+o(1)) \sqrt{n}$.

We consider the case when $\{S_i\}_{i\leq n}$ is a random walk on $\R$ with increments of mean zero and finite (positive) variance.
In this case, it is well known (e.g., using record times) that the length of the longest increasing subsequence satisfies $\E L\geq c\sqrt{n}$.
Our main result is an upper bound $\E L\leq n^{1/2 + o(1)}$, establishing the leading asymptotic behavior.
If $\{S_i\}_{i\leq n}$ is a simple random walk on $\Z$, we improve the lower bound by showing that $\E L \geq c\sqrt{n} \log{n}$.

We also show that if $\{\mathbf{S}_i\}$ is a simple random walk in $\Z^2$, then there is a subsequence of
$\{\mathbf{S}_i\}_{i\leq n}$ of expected length at
least $cn^{1/3}$ that is increasing in each coordinate. The above one-dimensional result yields an
upper bound of $n^{1/2 + o(1)}$. The problem of determining the correct exponent remains open.
\end{abstract}

\pagebreak

\section{Introduction}

For a function $S\colon \N\to \R$, its restriction to a subset $A$ of its
domain is denoted $S|_A$.  We say that $S|_A$ is {\bf increasing} if
$S(a)\leq S(b)$ for all $a,b\in A$ with $a\leq b$. Define
\[
\LIS(S|_{[0,n)})=\max\{|A|: A\subset [0,n), S|_A \text{ is increasing}\}.
\]
The main goal of this paper is to investigate $\LIS(S|_{[0,n)})$ when
$S\colon \N\to\Z$ is a random walk.  The simple random walk is the most natural
case, but our results apply to walks with steps of mean zero and finite
(positive) variance, that is, $S(n)=\sum_{i=1}^{n} X_i$ such that $X_i$ is an
i.i.d.\ sequence with $\E X_1=0$ and $0<\Var(X_1)<\infty$.  By normalising
$X_i$ we may clearly assume that $\Var(X_1)=1$.  We say that $S$ is the
simple random walk if $\P(X_1=1) = \P(X_1=-1) =1/2$.

The famous Erd\H{o}s-Szekeres Theorem \cite{ES} implies that $S|_{[0,n)}$
must contain either an increasing or a decreasing subsequence of length at
least $\sqrt{n}$.  This is sharp for general sequences, and it is easy to
see that there are even $n$-step simple walks on $\Z$ for which the longest increasing subsequence has length of order $\sqrt{n}$. By symmetry, increasing and decreasing subsequences have the same length distribution, but this does not immediately imply that a similar bound holds in high probability.

In random settings, there have been extensive studies of the longest
increasing subsequence in a uniformly random permutation $\sigma_n\in S_n$ initiated by Ulam~\cite{U}.
This is easily equivalent also to the case of a sequence $\{S(i)\}_{1\leq i \leq n}$ of i.i.d.\
(non-atomic) random variables.  A rich theory rose from the study of this
question, which is closely related to last passage percolation and other
models.  It was proved by Vershik and Kerov \cite{VK} and by Logan and
Shepp \cite{LS} that $\E \LIS(\sigma_n)=(2+o(1)) \sqrt{n}$ and $\LIS(\sigma_n)/\sqrt{n}\to 2$ in probability as $n\to \infty$. In this
case, much more is known. Baik, Deift and Johansson \cite{BDJ} proved that
the fluctuations of $\LIS(\sigma_n)$ scaled by $n^{1/6}$ converge to the
Tracy-Widom $F_2$ distribution, first arising in the study of the Gaussian
Unitary Ensemble. We refer the reader to Romik's book \cite{romik_LIS} for
an excellent survey of this problem.

On the other hand, it appears that this problem has not been studied so far
even for a simple random walk $S$. The expected length of the longest strictly
increasing subsequence of $S|_{[0,n)}$ is at most the expected size of the
range $S([0,n))$, hence is $O(\sqrt{n})$. Thus we consider (weakly) increasing subsequences.
Taking the set of record times, or alternatively the zero set of $S$ both
yield increasing subsequences of expected length $\Theta(\sqrt{n})$.  It is not
immediate how to do any better. The largest level set of $S$ still has size $\Theta(\sqrt{n})$.  See \cref{F:LIS}
for the longest increasing subsequence in one random walk instance.  Note
that the set of record times yields a similar lower bound for a general
random walk with mean zero and finite variance.

\begin{figure}
  \centering
  \includegraphics[width=0.6\textwidth]{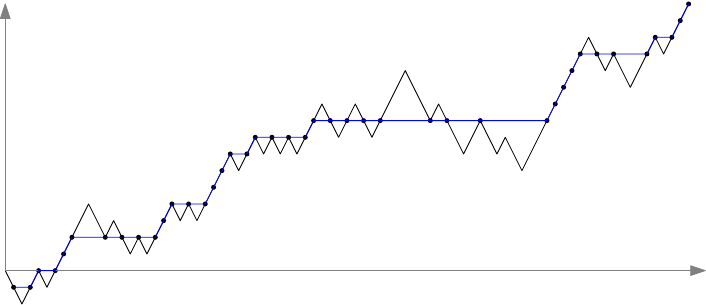}
  \caption{One increasing subsequence of maximal length in a simple random
    walk.}
  \label{F:LIS}
\end{figure}

On some reflection, one finds a number of arguments that yield the weaker
bound $\LIS(S|_{[0,n)}) \leq n^{3/4+\eps}$ for a simple random walk $S$.
For example, first one can show that with high probability, in every interval
$I\subset [0,n)$, each value $v$ is visited at most $C \sqrt{|I|} \log n$
times.  Assume $A\subset [0,n)$ is such that $S|_{A}$ is increasing. For
each $v\in S(A)$ define the interval $I_v=[a_v,b_v]$, where $a_v\in A$ is
the first (and $b_v\in A$ is the last) time $t\in A$ such that $S(t)=v$.
By monotonicity the intervals $I_v$ are disjoint. The length of the
subsequence is then bounded by $\sum_v C\sqrt{|I_v|} \log n$, where the
number of intervals is at most $R=|S([0,n))|$. As $R\leq n^{1/2+\eps}$ with
high probability, the Cauchy-Schwarz inequality gives the upper bound $C
\sqrt{n R} \log n\leq n^{3/4+\eps}$.  However, going beyond the exponent
$3/4$ requires more delicate arguments even in the case of a simple
random walk.

Although Ulam's problem and the problem of this paper are superficially
similar, it turns out that their structure is different, and finding
monotone subsequences in random walks is more closely related in nature to
restriction theorems for continuous functions.  In the continuous setting
the upper Minkowski dimension plays the role of counting.  Balka and Peres
\cite{BP} showed that a Brownian motion is not monotone on any set of
Hausdorff dimension greater than $1/2$.  However, working with the
Minkowski dimension requires understanding the structure of a set at a
specific scale, which is not needed for the Hausdorff dimension.  The proof
of \cite{BP} is based on Kaufman's uniform dimension doubling theorem for
two-dimensional Brownian motion. A key fact used there is that a closed set in $[0,1]$ of Hausdorff dimension strictly greater than 1/2 intersects the zero set of a Brownian motion with positive probability. Hausdorff dimension cannot be replaced here by upper Minkowski dimension.
Therefore, the methods developed in \cite{BP} are not powerful enough to prove the leading term of the upper bound even in the case of a simple
random walk.  However, Balka, M\'ath\'e and Peres analyzed the behaviour of
some self-affine functions motivated by \cite{BP} and questions of Kahane
and Katznelson \cite{KK}, which had a large impact on this paper. Later on,
M\'ath\'e and the authors of this paper \cite{ABMP} proved restriction
theorems for fractional Brownian motion using similar methods to this
paper's ones, and the case of self-affine functions was also handled in
\cite{ABMP}. In particular, they proved that a Brownian motion is not
monotone on any set of upper Minkowski dimension greater than $1/2$. Hence
the interaction between this paper and \cite{ABMP} played an important role
in both directions. For other restriction theorems in the continuous,
deterministic case, and for generic functions in the sense of Baire
category see Elekes \cite{E}, Kahane and Katznelson \cite{KK}, and
M\'ath\'e \cite{M}.

\bigskip

The main goal of this paper is to prove the following theorem. 


\begin{mainthm}\label{T:UB}
  Let $S(n)=\sum_{i=1}^{n} X_i$ be a random walk with i.i.d.\ steps
  satisfying  $\E(X_i)=0$ and
 $\Var(X_i) < \infty$.  For all $n$ large
  enough, for all $\ell \in \N^+$ we
 have
  \begin{equation*}
  \P\left(\LIS(S|_{[0,n)})\geq \ell \sqrt{n} 2^{5\sqrt{\smash[b]{\log_4 n (\log_4 \log_4 n)}}} \right)\leq
  \left(\frac{2}{\log_4 n}\right)^{2\ell}.
  \end{equation*}
  Moreover, if $E e^{t|X_1|}<\infty$ for some $t>0$ then we have
  \[
  \P\left(\LIS(S|_{[0,n)})\geq \ell \sqrt{n} 2^{5\sqrt{\smash[b]{\log_4 n (\log_4 \log_4 n)}}} \right) \leq
  n^{-\ell/2}.
  \]
\end{mainthm}

In the following corollaries let $S$ be a random walk as in Theorem~\ref{T:UB}.

\begin{corollary} \label{c:1}
  For all $\eps>0$ with probability $1-o(1)$ we have
  \[
  \LIS(S|_{[0,n)})\leq n^{1/2+\eps}.
  \]
\end{corollary}

\begin{corollary} \label{c:E}
For all $\varepsilon>0$ and $n$ large enough
$$\E \LIS(S|_{[0,n)}) \leq n^{1/2+\eps}.$$
\end{corollary}

In the other direction, we show that in the case of a simple random walk,
with high probability there are increasing subsequences somewhat longer
than the trivially found ones.

\begin{mainthm}\label{T:LB}
  Let $S\colon \N\to \Z$ be a simple random walk. For every $\eps>0$ for all $n$ large enough
  \[
  \P\left(\LIS(S|_{[0,n)}) < \eps \sqrt{n} \log_2 n\right) \leq 250\eps.
  \]
  Consequently, for all large enough $n$ we have
  \[
  \E \left(\LIS(S|_{[0,n)})\right) \geq (1/1000) \sqrt{n} \log_2 n.
  \]
\end{mainthm}

In \cref{s:high} we consider higher dimensional random walks.
Let $d\geq 2$ and let $S\colon \N \to \R^d$. We say that $S$ is \emph{increasing}
on a set $A\subset \N$ if all the coordinate functions of $S|_{A}$ are
non-decreasing, i.e.\ $S$ is increasing with respect to the coordinate-wise
partial order on $\R^d$.  Generalizing $\LIS$, we define
\[
\LIS(S|_{[0,n)})=\max\{|A|: A\subset [0,n), S|_A \text{ is increasing}\}.
\]
Since the restriction of a random walk to a single coordinate is again a
random walk, if $S$ is a $d$-dimensional random walk with mean $0$ and
bounded second moment then \cref{c:1} implies that $\LIS(S|_{[0,n)}) \leq
n^{1/2+o(1)}$ with probability $1-o(1)$. For a large class of two-dimensional
random walks we are able to prove a lower bound as well. However,
the problem of determining the correct exponent remains open.

\begin{mainthm}\label{T:2d_LB}
Let $S\colon \N\to \Z^2$ be a two-dimensional random walk with steps $X\in\R^2$ for which
\begin{itemize} \itemsep0pt
  \item the mean $\E X=\mathbf{0}$ is the zero vector,
  \item the covariance matrix $\mathrm{Cov}(X)=I_2$ is the identity matrix,
  \item the coordinates of $X$ have finite $2+\eta$ moments for some $\eta>0$.
\end{itemize}
Then there is a constant $c\in \R^+$ such that for eevery $\varepsilon>0$
and $n>0$
\[
  \P\left(\LIS(S|_{[0,n)}) < \eps n^{1/3}\right) \leq c\eps.
\]
Consequently, for all $n>0$ we have
$$\E \LIS(S|_{[0,n)}) \geq {\textstyle \frac{1}{4c}} n^{1/3}.$$
\end{mainthm}

Finally, in \cref{s:open} we state some open questions.

\subsection*{Acknowledgments}

OA thanks the organizers of the {\em probability, combinatorics and
geometry} meeting at the Bellairs Institute, as well as several
participants, in particular Simon Griffiths who proposed this problem, and
Louigi Addario-Berry, Guillaume Chapuy, Luc Devroye, G\'abor Lugosi and Neil
Olver for useful discussions.  The present collaboration took place mainly
during visits of OA and RB to Microsoft Research.  We are indebted to
Andr\'as M\'ath\'e and Boris Solomyak for useful suggestions.

\section{Upper bound}

To simplify notations in the proof, it is convenient to assume the length
of the random walk is a power of $4$.  By monotonicity in $n$ of
$\LIS(S|_{[0,n]})$, we can interpolate for all other $n$.  Our main result
\cref{T:UB} follows from the following by this monotonicity and the
substitution $4^n \to n$.


\begin{thm}\label{T:UB2}
  Let $S(n)=\sum_{i=1}^{n} X_i$ be a random walk with $\E(X_1)=0$ and
  $\Var(X_1)=1$.  For all $n$ large enough for all $\ell \in \N^+$ we
  have
  \begin{equation*}
  \P\left(\LIS(S|_{[0,4^n)})\geq \ell 2^{n+4\sqrt{\smash[b]{n\log_2 n}}} \right)\leq
  \left(\frac{2}{n^2}\right)^\ell.
  \end{equation*}
  Moreover, if $E e^{t|X_1|}<\infty$ for some $t>0$ then we have
  \[
  \P\left(\LIS(S|_{[0,4^n)})\geq \ell 2^{n+4\sqrt{\smash[b]{n\log_2 n}}} \right)\leq
  2^{-\ell n}.
  \]
\end{thm}

The main goal of this section is to prove \cref{T:UB2}.
The key is a
multi-scale argument, the time up to $4^n$ is split into $4^k$ intervals.
We consider the number of these intervals that intersect our set $A$, as well as the sizes of
intersections.  Repeating this allows us to get (inductively) better and
better bounds. The dependence on the randomness of the walk is done through
some estimates on the local time, which we derive in the following
subsection.

Throughout this section, fix a random walk $S(n)=\sum_{i=1}^{n}X_i$ with
$\E X_1=0$ and $\Var(X_1)=1$.  Various constants below depend only on the
law of $X_i$.  We will use the following theorems in this section.

\begin{theorem}[Petrov,~\cite{P}] \label{t:P}
  There is a constant $c$ such that for all $n\in \N^+$ and $\lambda\geq 0$
  we have
  \[
  \sup_{x\in \R} \P(x\leq S(n)\leq x+\lambda) \leq
  c \frac{\lambda+1}{\sqrt{n}}.
  \]
\end{theorem}

For the following theorems see \cite[Thm.~A.2.5]{LL} and its corollaries.

\begin{theorem} \label{t:lambda1}
  For all $n\in \N^+$ and $\lambda>0$ we have
  \[
  \P\left(\max_{0\leq i\leq n} |S(i)|\geq \lambda \sqrt{n}\right) \leq
  \frac{1}{\lambda^2}.
  \]
\end{theorem}

\begin{theorem}\label{t:lambda2}
  Assume that $\E e^{t|X_1|}<\infty$ for some $t>0$. Then there is a
  constant $c>0$ such that for all $n\in \N^+$ and $0\leq \lambda\leq
  \sqrt{n}$ we have
  \[
  \P\left(\max_{0\leq i\leq n} |S(i)|\geq \lambda \sqrt{n}\right) \leq
  e^{-c\lambda^2}.
  \]
\end{theorem}

\subsection{Scaled local time estimates}

\begin{defn}
  Let $m,p\in \N$ and $q\in \Z$. A {\bf time interval of order $m$} is of
  the form
  \[
  I_{m,p} = [p 4^m, (p+1)4^m)\subset \N.
  \]
  A {\bf value interval of order $m$} is of the form
  \[
  J_{m,q} = [q 2^m, (q+1) 2^m)\subset \R.
  \]
  Note that a time interval is a subset of $\N$, while a value interval is
  a real interval.  For all $0\leq k\leq m$ let $\iI_{m,k,p}$ be the set of
  time intervals of order $m-k$ contained in $I_{m,p}$.  Clearly
  $|\iI_{m,k,p}|=4^k$.
\end{defn}

\begin{defn}
  The {\bf scaled local time} $S_{m,k,p,q}$ is the number of order $m-k$
  intervals in $\iI_{m,k,p}$ in which $S$ takes at least one value in
  $J_{m-k,q}$:
  \[
  S_{m,k,p,q} = \left|\{I\in\iI_{m,k,p} : \exists x\in I, ~ S(x)\in
  J_{m-k,q}\}\right|.
  \]
\end{defn}

Our intermediate goal is to prove the following uniform estimate on scaled
local times.

\begin{prop}\label{P:local_time_bound}
  There is a $\gamma\geq 2$ such that for all $n$ large enough
  \[
  \P\left(S_{m,k,p,q} \leq \gamma n2^k \textrm{ for all } k\leq m\leq
    n,\, p< 4^{n-m},\, q\in \Z \right) \geq 1 - 2^{-(n+1)}.
  \]
\end{prop}

We begin with an estimate on the expectation of a single scaled local time.
Let $\E_x$ and $\P_x$ denote the expectation and probability for a random
walk started at $x$.

\begin{lemma}\label{L:ES}
  For some absolute constant $c$ and for all $x,m,k,p,q$ we have
  \[
  \E_x S_{m,k,p,q} \leq c 2^k.
  \]
\end{lemma}


\begin{proof}
  The proof uses the idea that conditioned on the event that some time
  interval contributes to $S_{m,k,p,q}$, with
 probability bounded from $0$,
  the random walk is still nearby at the end of the time interval.

  The strong Markov property of the walk $S$ and translation invariance
  imply that it is enough to
consider the case $p=q=0$.  Since the central
  limit theorem yields
  $\lim_{j\to \infty} \P_0(0\leq S(j)\leq \sqrt{j})>0$, there exist $c_1>0$
  and $N\in \N$ such that for all $j\geq 1$ we have
  \begin{equation*}
    \P_0(0\leq S(j)\leq N\sqrt{j})\geq c_1.
  \end{equation*}
  For each $1\leq i\leq 4^k$ let $A_i$ be the event that
  $S(i4^{m-k})\in \bigcup_{j=0}^{N} J_{m-k,j}$ and let $B_i$ be the event

that $I_{m-k,i-1}$ counts towards $S_{m,k,0,0}$.  By the above inequality,
  for all $i$ we have

\begin{equation*}
    \P_x(A_i \,|\, B_i)\geq c_1.
  \end{equation*}
  By \cref{t:P} there exists a constant $c_2$ such that for all $x\in \R$
  and $i$
we have
  \begin{equation*}
    \P_x(A_i) \leq \frac{c_2(N+1) 2^{m-k}}{\sqrt{i 4^{m-k}}}=\frac{c_2(N+1)}{\sqrt{i}}.
  \end{equation*}
The above two inequalities imply that
\begin{equation*} \E_x S_{m,k,0,0} =\sum_{i=1}^{4^k} \P_x(B_i)\leq \sum_{i=1}^{4^k}
\frac{c_2(N+1)}{c_1\sqrt{i}} \leq c 2^k,
\end{equation*}
where $c=2c_2(N+1)/c_1$. The proof is complete.
\end{proof}

Next we estimate the tail of a single scaled local time.

\begin{lemma}\label{L:s_tail}
  There is an absolute constant $C$ such that for all $x,m,k,p,q$ and
  $\ell\in \N^+$
  we have
  \[
  \P_x\left(S_{m,k,p,q} \geq C \ell 2^k \right) \leq 2^{-\ell}.
  \]
\end{lemma}

\begin{proof}
  Let $C=\lceil 2c \rceil$, where $c$ is the constant of \cref{L:ES}
  and $\lceil\cdot\rceil$ denotes rounding up.  By Markov's inequality we
  have $\P_x(S_{m,k,p,q} \geq C 2^k) \leq 1/2$, establishing the claim
  for $\ell=1$.  We proceed inductively: Assume that the claim holds for
  some $\ell\geq 1$.  Observe the walk starting at time $p 4^m$ either
  until we reach time $(p+1)4^m$ or until $C\ell 2^k$ sub-intervals of
  order $m-k$ contribute to $S_{m,k,p,q}$.  The latter happens with
  probability at most $2^{-\ell}$.  By the strong Markov property the
  conditional probability that there are $C 2^k$ additional sub-intervals
  contributing to $S_{m,k,p,q}$ is at most $1/2$, proving the claim
  for $\ell+1$.
\end{proof}

\begin{proof}[Proof of \cref{P:local_time_bound}]
  Let $\gamma=7C$, where $C$ is the constant of \cref{L:s_tail}. We
  apply \cref{L:s_tail} with $\ell=7n$ to each of the relevant $m,k,p,q$.
  Since $0\leq k\leq m\leq n$, there are $n+1$ choices for each of $m$ and
  $k$. As $p\in[0,4^{n-m})$, there are at most $4^n$ options for $p$.
  If $\max_{0\leq i\leq 4^n} |S(i)|<2^{2n}$ then $q$ with $|q|\geq
  2^{2n}+1$ have scaled local time $0$.  This is likely, as
  \cref{t:lambda1} yields that $\P(\max_{0\leq i\leq 4^n} |S(i)|\geq
  2^{2n})\leq 2^{-2n}$.  These imply that for
  all $n$ large enough we have
  \begin{align*}
    \P\left( \exists m,k,p,q, \text{ s.t. } S_{m,k,p,q} > 7C n 2^k \right)
    &\leq (n+1)^2 4^n2^{2n+2} 2^{-7n}+2^{-2n}\\
    &\leq  2\cdot 2^{-2n}\leq 2^{-(n+1)}.
  \end{align*}
  Clearly we may also require $\gamma\geq 2$.
\end{proof}

\subsection{No long increasing subsequence}

Next, we use \cref{P:local_time_bound} to rule out the existence of very long
increasing subsequences in the random walk. We need the following definition.

\begin{defn}
  Let $S$ be a function and let $A=\{a_1,\dots,a_k\}$ be a finite set such
  that $a_1<a_2<\dots <a_k$.  The {\bf variation} of $S$ restricted to $A$
  is defined as
  \[
  V^{1}(S|_A)= \sum_{i=1}^{k-1} |S(a_{i+1})-S(a_i)|.
  \]
\end{defn}

Note that if $S|_A$ is increasing then $V^{1}(S|_A)$ equals the
diameter of $S(A)$. The upper bound of \cref{T:UB2} follows from the
following proposition.

\begin{prop}\label{P:n_product}
  Fix $n=ab$ with $a,b\in \N^+$. Assume that a walk $S\colon \N \to \R$ is such that
  \begin{enumerate}
  \item \label{as:1} the event of \cref{P:local_time_bound} occurs,
  \item \label{as:2} $\max_{0\leq i\leq 4^n} |S(i)|\leq n 2^n$.
  \end{enumerate}
  Then we have
  \[
  \LIS(S|_{[0,4^n)}) \leq (\gamma n 2^{b+1})^{a+1}.
  \]
\end{prop}

\begin{proof}
  Let $A\subset[0,4^n)$ be a set such that $S|_A$ is increasing.  For
  $0\leq \ell \leq a$ let
  \[
  D_\ell = \{I\in \iI_{ab,\ell b,0}: I\cap A\neq \emptyset\}  \quad
  \textrm{and} \quad  d_\ell = \frac{|D_\ell|}{(\gamma n 2^{b+1})^\ell}
  \]
  be the set of intervals of order $(a-\ell)b$ that intersect $A$, and its
  size with a convenient normalization.  Clearly $D_a=A$ and $d_0=|D_0|=1$.
  In order to prove the claim we prove inductively bounds on $|D_\ell|$.

  Let $\ell\geq 1$ and index the elements of $D_{\ell-1} =
  \{I_1,I_2,\dots\}$, and suppose that interval $I_i$ contains $p_i$
  intervals in $D_\ell$, so that $|D_\ell| = \sum_{i=1}^{|D_{\ell-1}|}
  p_i$. By assumption~\eqref{as:1} for all $q$ we have that
  $J_{(a-\ell)b,q}$ is visited in at most $\gamma n 2^b$ sub-intervals of
  $I_i$. It follows that if $p_i>\gamma n 2^b$ then $S|_{A\cap I_i}$ must
  visit at least $p_i/(\gamma n 2^b)$ value intervals of order
  $(a-\ell)b$.  The diameter of the union of
  these visited intervals is at least $(p_i/(\gamma n
  2^b)-2)2^{(a-\ell)b}$. This leads to a variation bound
  \begin{equation}\label{eq:var}
    V^{1}(S|_{A\cap I_i}) \geq \left(\frac{p_i}{\gamma n 2^b} - 2\right)
    2^{(a-\ell)b}.
  \end{equation}
Assumption~\eqref{as:2} yields that $V^{1}(S|_A) \leq
  2(n 2^n)\leq \gamma n 2^{n}$. Thus
  \begin{equation}\label{eq:Ii}
    \sum_{i=1}^{|D_{\ell-1}|} V^{1}(S|_{A\cap I_i}) \leq \gamma n 2^n.
  \end{equation}
  Inequalities \eqref{eq:var} and \eqref{eq:Ii} and $n=ab$ imply that
  \[
    \gamma n 2^n \geq \sum_{i=1}^{|D_{\ell-1}|} \left(\frac{p_i}{\gamma n
        2^b} - 2\right)2^{(a-\ell)b}=2^{n-\ell b}
    \left(\sum_{i=1}^{|D_{\ell-1}|} \frac{p_i}{\gamma n 2^b} -
      2|D_{\ell-1}| \right),
    \]
  and therefore
  \begin{equation}\label{eq:D}
    \sum_{i=1}^{|D_{\ell-1}|} \frac{p_i}{\gamma n 2^b} - 2|D_{\ell-1}| \leq
    \gamma n 2^{\ell b}.
  \end{equation}
Using $|D_\ell| = \sum_{i=1}^{|D_{\ell-1}|} p_i$ and dividing \eqref{eq:D} by $2(\gamma n
  2^{b+1})^{\ell-1}$ yields
  \[
  d_\ell - d_{\ell-1} \leq \frac{\gamma n 2^b}{2(2\gamma n)^{\ell-1}} \leq
  \gamma n 2^{b-\ell},
  \]
  where we have used that $\gamma n \geq 1$. As $d_0=1$, the above
  inequality implies
  \[
  d_\ell \leq 1 + \sum_{i=1}^{\ell} \gamma n 2^{b-i} \leq \gamma n 2^{b+1}
  \]
  for every $\ell\leq a$. In particular we get for $\ell=a$
  \[
  |A| = |D_a| = (\gamma n 2^{b+1})^a d_a \leq (\gamma n 2^{b+1})^{a+1}. \qedhere
  \]
\end{proof}

Finally, we use \cref{P:n_product} to derive an estimate on the likelihood
of long increasing subsequences in a random walk.

\begin{proof}[Proof of \cref{T:UB2}]
  First we prove the theorem for $\ell=1$. Let
  \[
  a = \left\lceil \sqrt{n/\log_2 n}  \right\rceil 
  \quad \text{and} \quad
  b = \left\lceil \sqrt{n\log_2 n} \right\rceil,
  \]
  where $\lceil\cdot\rceil$ denotes rounding up. Note that $ab\geq n$. We
  consider $S$ up to time $4^{ab}$.  For $n$ large enough with probability
  $1-2^{-(ab+1)}\geq 1-2^{-(n+1)}$ the event of \cref{P:local_time_bound}
  occurs for $n'=ab$. Moreover, \cref{t:lambda1} implies that
  \[
  \P\left(\max_{i\leq 4^{ab}} \{ |S(i)| \} \geq ab 2^{ab}\right) \leq
  \frac{1}{(ab)^2}\leq \frac{1}{n^2}.
  \]
  Thus with probability at least $1-2^{-(n+1)}-1/n^2\geq
  1-2/n^2$ the conditions and conclusion of \cref{P:n_product} hold
  for $ab$.

  \medskip

  Suppose additionally that $E e^{t|X_1|}<\infty$ for some $t>0$.
  Then \cref{t:lambda2} yields that for some constant $c>0$
  and for all $n$ large enough
  \[
  \P\left(\max_{i\leq 4^{ab}} \{ |S(i)| \}  \geq ab 2^{ab}\right)
  \leq e^{-c(ab)^2}\leq 2^{-(n+1)}.
  \]
  Thus with probability at least $1-2^{-n}$ the conditions and conclusion
  of \cref{P:n_product} hold for $ab$.

  Let $n$ be such that \cref{P:n_product} holds for $ab$.  Since
  $\LIS(S|_{[0,4^n)})$ is increasing in $n$, we obtain that for $n$ large
  enough
  \begin{align*} \label{eq:S04}
    \LIS\big(S|_{[0,4^n)}\big) \leq
    \LIS\big(S|_{[0,4^{ab})}\big)
    & \leq \big(\gamma ab 2^{b+1}\big)^{a+1} = 2^{ab} 2^b \big(2\gamma
    ab\big)^{a+1} \notag \\
    & \leq 2^{n + 3\sqrt{\smash[b]{n\log_2 n}} + O (\sqrt{\smash[b]{n/\log_2 n}})} \\
    & < 2^{n+4\sqrt{\smash[b]{n\log_2 n}}} - 1.
  \end{align*}

  This proves \cref{T:UB2} if $\ell=1$. For the general case fix $n,N\in
  \N^+$, it is enough to prove that for all $\ell \in \N^+$ we have
  \begin{equation} \label{eq:PS}
    \P\left(\LIS(S|_{[0,4^n)})\geq \ell N \right)
    \leq \P\left(\LIS(S|_{[0,4^n)})\geq N \right)^{\ell},
  \end{equation}
  then setting $N=\lfloor 2^{n+4\sqrt{\smash[b]{n\log_2 n}}} \rfloor$ concludes the
  proof, where $\lfloor \cdot \rfloor$ denotes rounding down.  Let
  $T_0=0$.  If $T_i$ is already defined then let $T_{i+1}$ be the
  minimal integer $t$ so that $\LIS(S|_{[T_i,t)}) \geq N$. Since
  $\LIS(S|_{[T_i,t)})$ increases by at most $1$ when incrementing $t$,
  we actually have $\LIS(S|_{[T_i,T_{i+1})}) = N$.  By the strong
  Markov property at $T_i$, we see that $T_{i+1}-T_i$ are i.i.d.\
  copies of $T_1$.  However, $\LIS(S|_{[0,4^n)}) \geq \ell N$ requires
  $T_{i+1}-T_i\leq 4^n$ for all $0\leq i\leq \ell-1$, with
  probability at most $\P\left(\LIS(S|_{[0,4^n)})\geq N \right)^{\ell}$.
  This implies \eqref{eq:PS}, and the proof is complete.
\end{proof}

\section{Lower bound for a simple random walk}

The goal of this section is to prove \cref{T:LB}.  For simplicity, we
present our argument only for the simple random walk on $\Z$.  However, it
seems that the argument should extend with minor changes to any random walk
with bounded integer steps of 0 mean, and finite variance.  The construction
relies on values appearing multiple times in the walk, and fails more
fundamentally if the walk is not supported on multiples of some $\alpha$.

\begin{defn}
  Let $\tau_n$ denote the {\bf hitting time} of $n$ by the simple random
  walk.  Let $\ord_2(x)$ be the {\bf 2-order} of $x\in\Z\setminus \{0\}$, that is, the
  number of times it is divisible by $2$.
\end{defn}

\begin{lemma}\label{L:excursions}
  Consider a simple random walk from $x-s$ conditioned to hit $x+s$ before
  returning to $x-s$, and stopped when it reaches $x+s$. Let $a,b$ be the
  times of the first and last visits to $x$.  Then:
  \begin{enumerate}
  \item \label{i:1} The number of visits to $x$ is geometric with mean $s$.
  \item \label{i:2} The walk on $[0,a]$ is a walk conditioned to hit $x$
    before returning to $x-s$, and stopped when it reaches $x$.
  \item \label{i:3} The walk on $[b,\tau_{x+s}]$ is a walk from $x$
    conditioned to hit $x+s$ without returning to $x$, and stopped when it
    reaches $x+s$.
  \item \label{i:4} The two sub-walks and the geometric variable are
    independent.
  \end{enumerate}
\end{lemma}

\begin{proof}
  In order to prove the first statement we first consider a simple random
  walk from $x$ up to the time $\tau$ when it reaches either $x-s$ or
  $x+s$.  This walk has probability $(s-1)/s$ of returning to $x$
  without hitting $\{x-s,x+s\}$, at which time another excursion from $x$
  begins.  Therefore the number of visits to $x$ on $[0,\tau]$ is geometric
  with mean $s$.  Moreover, the number of visits to $x$ is independent of
  whether the walk hits $x+s$ or $x-s$, so when conditioning on hitting
  $x+s$ the distribution is still geometric with mean $s$, which proves the
  first statement.

  Now we return to our original walk from $x-s$. Excursions from $x$ either
  return to $x$, or hit $x+s$, or hit $x-s$.  The partition into excursions
  around $x$ does not give any information on the trajectory within each
  excursion, except for its type, and the other claims follow.
\end{proof}

\begin{lemma}\label{L:LB_stopped}
  Let $S\colon \N\to \R$ be a simple random walk. For all $n$ we have
  \[
  \E \LIS\left(S|_{[0,\tau_{2^n})}\right) \geq n 2^{n-1},
  \]
  and for any $\eps>0$,
  \[
  \P\left(\LIS(S|_{[0,\tau_{2^n})}) < (1-\eps)n 2^{n-1}\right) \leq
  \frac{2}{\eps^2 n^2}.
  \]
\end{lemma}

\begin{proof}
  We construct an increasing subsequence of $S|_{[0,\tau_{2^n})}$ as
  follows.  Informally, we take some times $i$ to be in our index set,
  greedily in decreasing order of the 2-order of $S(i)$.

  For each integer $0\leq x\leq 2^n$ we construct an interval $I_x =[a_x,b_x]
  \subset [0,\tau_{2^n}]$.  The intervals are such that if $x<y$ then $b_x
  < a_y$.  Given such intervals, we have that $S$ is increasing along
  $A\subset [0,\tau_{2^n})$, where
  \[
  A=\bigcup_{x=1}^{2^n-1}\{i\in I_x : S(i)=x \}.
  \]

  We start by setting $I_0 = [0,b_0]$ and $I_{2^n} = [\tau_{2^n},\tau_{2^n}] =
  \{\tau_{2^n}\}$, where $b_0$ is the last visit to $0$
  before $\tau_{2^n}$. Let $k\leq n-1$ and $0<x<2^n$ be such that $\ord_2(x)=k$
  and assume by induction that $I_y=[a_y,b_y]$ are already defined for all
  $0<y\leq 2^n$ for which $\ord_2(y)>k$.  Now we define $I_x$. Let
  $\underline{x}=x-2^k$ and $\overline{x} = x+2^k$, then clearly
  $\ord_2(\underline{x}),\ord_2(\overline{x}) \geq k+1$. Thus
  $I_{\underline{x}}=[a_{\underline{x}},b_{\underline{x}}]$ and
  $I_{\overline{x}}=[a_{\overline{x}},b_{\overline{x}}]$ are already
  defined by the inductive hypothesis.  Let $I_x=[a,b]$, where $a$ is the
  first hitting time of $x$ after $b_{\underline{x}}$ and $b$ is the time
  of the last visit to $x$ before $a_{\overline{x}}$.  See \cref{F:dyadic}
  for an example.

\begin{figure}
\centering
\includegraphics[width=0.6\textwidth]{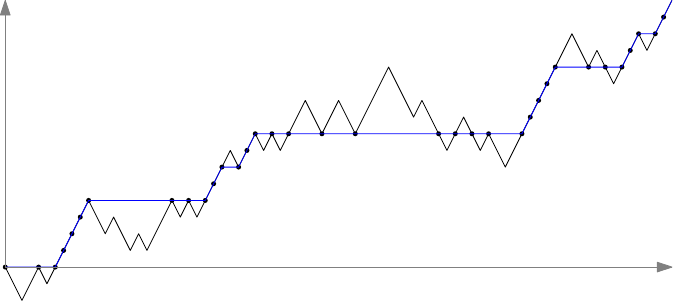}
\caption{The increasing subsequence constructed for
\cref{L:LB_stopped} in a simple random walk stopped at $16$. All visits to
$8$ are used, then all compatible visits to $4$,$12$, followed by
$2,6,10,14$ and a single visit to each odd value.  There exist longer
subsequences of length $42$ in this case.}
\label{F:dyadic}
\end{figure}

  Assuming $\ord_2(x)=k$, we show that the law of $S$ restricted to
  $[b_{\underline{x}},a_{\overline{x}}]$ is that of a simple random walk
  started at $x-2^k$ conditioned to hit $x+2^k$ before returning to $x-2^k$,
  and stopped when hitting $x+2^k$. This is seen inductively using
  \cref{L:excursions}, and since the walk after the last visit to
  $0$ before $\tau_{2^n}$ cannot return to $0$.

From the above, we deduce that for all $x\in(0,2^n)$ with $\ord_2(x)=k$
the number of visits to $x$ in $I_x$ is geometric with mean $2^{k}$, and
these geometric variables are all independent.
Since for each $k\in \{0,\dots, n-1\}$ there are $2^{n-k-1}$ integers
$x\in(0,2^n)$ with $\ord_2(x)=k$,
we get that
\[
\E|A| =\sum_{k=0}^{n-1} 2^{n-k-1} 2^{k}=n2^{n-1}.
\]

  As any geometric $X$ satisfies $\Var X=(\E X)(\E X-1)$ and our geometric
  random variables are all independent, we obtain that
  \[
  \Var |A| = \sum_{k=0}^{n-1} 2^{n-k-1} 2^{k} \left(2^{k}-1\right)
           \leq \sum_{k=0}^{n-1} 2^{n+k-1}
           \leq 2^{2n-1}.
  \]

The second claim now follows by Chebyshev's inequality.
\end{proof}

\begin{proof}[Proof of \cref{T:LB}]
  Fix $\eps>0$.  For large enough $n\in \N$ let $m=m(n)$ be an integer such
  that
  \begin{equation} \label{eq:mn}
  \frac 19 m 2^m \leq \eps \sqrt{n} \log_2 n < \frac 14 m 2^{m}.
  \end{equation}
  Then we have that
  \begin{equation}
    \label{eq:total}
    \P\left(\LIS(S|_{[0,n)}) < \eps \sqrt{n}\log_2 n\right)
    \leq \P\left(\LIS(S|_{[0,\tau_{2^m})}) < \tfrac12 m 2^{m-1} \right) +
    \P\big(\tau_{2^m}> n\big).
  \end{equation}
  Applying \cref{L:LB_stopped} for this $m$ we get
  \[
  \P\left(\LIS(S|_{[0,\tau_{2^m})}) < \tfrac12 m 2^{m-1} \right) \leq
  \frac{8}{m^2}.
  \]
  Moreover, \cite[Thm.~2.17]{LPW} and \eqref{eq:mn} imply that
  \[
  \P\left(\tau_{2^m} > n\right) \leq \frac{12 \cdot 2^m}{\sqrt{n}} \leq
  108 \eps \frac{\log_2 n}{m}.
  \]
  Since $\log_2 n/m \to 2$ and $8/m^2 \to 0$ as $n\to
  \infty$, plugging the previous bounds in \eqref{eq:total} gives for $n$
  large enough
  \[
  \P\left(\LIS(S|_{[0,n)}) < \eps \sqrt{n}\log_2 n\right) \leq
  \frac{8}{m^2} + 108\eps \frac{\log_2 n}{m} \leq 250 \eps.
  \]

Finally, applying the above inequality for $\eps=1/500$ implies that
  \[
  \E \left(\LIS(S|_{[0,n)})\right) \geq \frac{\sqrt{n} \log_2 n}{1000}.
  \qedhere
  \]
\end{proof}


\section{Random walks in higher dimensions} \label{s:high}

The main goal of this section is to prove \cref{T:2d_LB}.  As noted, the
upper bound in the one-dimensional case holds trivially in every dimension.
For sequences $\{a_n\},\{b_n\}$ we use the notation $a_n\sim b_n$ if
$a_n/b_n\to 1$ as $n\to \infty$.  The lower bound is based on the
following estimate by Denisov and Wachtel, see \cite[Example~2]{DW} and see
there the history of similar estimates for Brownian motion and random
walks.

\begin{theorem}\label{T:2dgreedy}
  Let $S\colon \N\to \R^2$ be a two-dimensional random walk satisfying the
  conditions of \cref{T:2d_LB}.  Let $\tau$ be the hitting time of the
  positive quadrant: $\tau=\inf\{n>0 : S(n) \in \R_+^2\}$. Then there is
  some $c\in \R^+$ so that
  \[
  \P(\tau>n) \sim c n^{-1/3}.
  \]
\end{theorem}

More generally, for a higher dimensional random walk $S\colon \N\to\R^d$, define
the hitting time
\[
\tau =\inf\{n>0 : S(n) \in \R_+^d\}.
\]
Denisov and Wachtel \cite[Theorem~1]{DW} proved that $\P(\tau>n)\sim cn^{-\alpha}$
for some $c\in \R^+$ and $\alpha \in (0,\infty)$, where $\alpha$ is the
exponent corresponding to Brownian motion staying outside a quadrant up to
time $t$ (assuming again that the walk is normalized so that $\E X=\mathbf{0}$ and
$\E X_i X_j = \delta_{ij}$, and that $\E\|X\|^{2+\eps}<\infty$ for some
$\eps>0$).  Consequently, the following lemma completes the proof of
\cref{T:2d_LB} (with $\alpha=1/3$), and gives a similar lower bound
for random walks in higher dimensions.

\begin{lemma}
  Let $S\colon \N\to\R^d$ be a random walk in $\R^d$, and let $0<\alpha<1$ be
  such that
  \[
  \P(\tau>n) = O(n^{-\alpha}).
  \]
  Then there is a constant $c\in \R^+$ such that for all $\eps>0$
  and $n>0$
  \begin{equation} \label{eq:2dim}
    \P(\LIS(S|_{[0,n)}) < \eps n^{\alpha})\leq c\eps.
  \end{equation}
  Consequently, for all $n>0$ we have
  \begin{equation*}
    \E \LIS(S|_{[0,n)})\geq {\textstyle \frac{1}{4c}} n^{\alpha}.
  \end{equation*}
\end{lemma}

\begin{proof}
  Fix $\eps>0$. Define the greedy increasing subsequence with time indices
  given by the recursion
  \[
  a_0=0, \quad a_{i+1}=\min\{a>a_i: S(a)-S(a_i)\in \R_+^d\}.
  \]
  Setting $k_n = \lfloor \eps n^\alpha\rfloor$, we see that if $a_{k_n} <
  n$ then $\LIS(S|_{[0,n)}) \geq k_n+1 > \eps n^{\alpha}$.  This gives a set $\{a_i:
  i\in \N\}\subset \N$ with i.i.d.\ increments $X_n = a_n-a_{n-1}$ with the
  law of $\tau$.

  Choose $c_1\in \R^+$ such that for all $n\in \N^+$
  \[
  \P(\tau>n)\leq c_1 n^{-\alpha},
  \]
  and define the truncated variables $Y_i = X_i \mathbf{1}\{X_i\leq
  n\}$. Then
  \begin{equation}\label{eq:xiyi}
    \P\left(\exists i\leq k_n: X_i\neq Y_i \right) \leq k_n \P(\tau>n) \leq
    c_1\eps.
  \end{equation}

  The $Y_i$ also form an i.i.d.\ sequence and satisfy
  \begin{equation*}
    \E Y_i \leq \sum_{m=0}^{n} \P(\tau>m)\leq c_2 n^{1-\alpha},
  \end{equation*}
  where $c_2\in \R^+$ depends only on $\alpha$ and $c_1$. By Markov's
  inequality,
  \begin{equation}\label{eq:markov}
    \P\left(\sum_{i=1}^{k_n} Y_i \geq n \right) \leq \frac{k_n \E Y_1}{n} \leq
    c_2 \eps.
  \end{equation}
  Combining \eqref{eq:xiyi} and \eqref{eq:markov} we obtain
  \begin{align*}
    \P\bigg(\sum_{i=1}^{k_n} X_i \geq n\bigg) &\leq \P\Big(\exists i\leq k_n:
    X_i\neq Y_i \Big)+\P\bigg(\sum_{i=1}^{k_n} Y_i \geq n \bigg) \\
    &\leq (c_1+c_2)\eps.
  \end{align*}
  As noted, this is a bound on $\P(\LIS(S|_{[0,n)}) < \eps n^{\alpha})$.
  Hence \eqref{eq:2dim} holds with $c=c_1+c_2$.  Applying \eqref{eq:2dim}
  with $\eps=1/(2c)$ yields the second claim.
\end{proof}


\section{Open Questions} \label{s:open}

There are many potential extensions of our results.  Two central open
problems are to reduce the gap between the lower and upper bounds in
dimension one, and to determine the right order of magnitude in higher
dimensions.  Moreover, our lower bound in \cref{T:LB} is specific to the
simple random walks,
and our proof does not work for general random walks.

\begin{question}
  Let $S\colon \N\to \R$ be a random walk with zero mean and finite (positive)
  variance.  Is there a constant $a$ such that, with probability $1-o(1)$,
  \[
  \LIS(S|_{[0,n)})\leq \sqrt{n} \log^a n?
  \]
  Does this upper bound hold at least when $S$ is a simple random walk?
\end{question}

\begin{question}
  Let $d\geq 2$ and let $S^d\colon  \N \to \R^{d}$ be a $d$-dimensional simple
  random walk. What is the order of magnitude of $\LIS(S^d|_{[0,n)})$ with
  probability $1-o(1)$?
\end{question}

A greedy construction gives (for the simple random walk) $\E
\LIS(S^d|_{[0,n)}) \geq n^{c_d+o(1)}$ for some $c_d>0$, with $c_2=1/3$.
Can this be improved?  Can we find an upper bound of the form
$n^{C_d+o(1)}$, with $C_d\to 0$ as $d\to \infty$?


\end{document}